\documentclass[12pt]{article}
\usepackage[utf8]{inputenc}
\usepackage[margin=1in]{geometry}
\usepackage{amsmath}
\usepackage{amssymb}             
\usepackage{algorithm2e}
\usepackage{enumerate}
\usepackage{amsthm}
\usepackage{comment}
\usepackage{appendix}
\usepackage{subcaption}
\usepackage{tikz-cd}
\RestyleAlgo{ruled}
\usepackage{url}
\usepackage{xcolor}
\usetikzlibrary{shadows}
\makeatother
\usepackage{graphicx}

\newcommand{\Z}{\mathbb{Z}}

\newcommand\vr{\mathcal{V}}

\newcommand\sort{\mathrm{sort}}
\newcommand\ph[1]{\text{PH}_{#1}}
\newcommand\lune{\mathrm{lune}}

\newcommand\rng{\mathrm{RNG}}
\newcommand\rnc{\mathrm{RNC}}

\newcommand\diam{\mathrm{diam}}
\newcommand\crit{\mathrm{Crit}}
\newcommand\match{\mathrm{Match}}
\newcommand\supp{\mathrm{Supp}}
\newcommand\crng{\mathrm{CRNG}}
\newcommand\crnc{\mathrm{CRNC}}
\newcommand\reach{\mathrm{reach}}

\newcommand\rvr{\mathcal{R}}

\newcommand\dvr{\mathcal{D}}

\theoremstyle{plain}
\newtheorem*{theorem*}{Theorem}
\newtheorem{theorem}{Theorem}[section]
\newtheorem{lemma}[theorem]{Lemma}

\theoremstyle{definition}
\newtheorem{definition}[theorem]{Definition}

\newtheorem{example}[theorem]{Example}
\newtheorem{remark}[theorem]{Remark}
\newtheorem{conjecture}[theorem]{Conjecture}

\title{The distilled Vietoris-Rips filtration for persistent homology and a new memory-efficient algorithm. }

\author{Musashi Ayrton Koyama, Vanessa Robins, Katharine Turner}
\date{April 2024}

\begin{document}

\maketitle

\begin{abstract}
The long computational time and large memory requirements for computing Vietoris-Rips persistent homology from point clouds remains a significant deterrent to its application to `big data'. This paper aims to reduce the memory footprint of these computations. It presents a new construction, the \emph{distilled Vietoris-Rips filtration}, and proves that its persistent homology is isomorphic to that of standard Vietoris-Rips. The distilled complex is constructed using a discrete Morse vector field defined on the reduced Vietoris-Rips complex. The algorithm for building and reducing the distilled filtration boundary matrix is highly parallelisable and memory efficient. It can be  implemented for point clouds in any metric space given the pairwise distance matrix.   
\end{abstract}

\section{Introduction}

Computing persistent homology on large data sets has been, and still remains one of the most significant bottlenecks to the adoption of persistent homology by the wider scientific community. While there are software packages which can compute persistent homology for large point clouds, usually these come with some sort of restriction on the data-set. 
For example, persistent homology of the Cech filtration is efficiently computed using alpha-shapes for point clouds in low-dimensional Euclidean spaces. 

The Vietoris-Rips filtration is a natural choice of data structure for computing persistent homology when working with arbitrary finite metric spaces.  
This is because constructing the Vietoris-Rips filtration only requires the pairwise distance matrix. However, whilst applying Vietoris-Rips persistent homology to finite metric spaces is \emph{conceptually} simple, applying it in practice to finite metric spaces of even a modest size presents difficulties. Ripser \cite{Ripser}, along with its GPU accelerated counterpart \cite{zhang2020gpu} are the current state of the art software for computing Vietoris-Rips persistent homology on finite metric spaces. With that said, computing $\ph{1}(X)$ for $|X| = 10^5$ and above remains a computation limited to the realm of supercomputers and is currently not even remotely feasible on a civilian machine. 

There are currently two main roadblocks to computing Vietoris-Rips persistent homology on larger point clouds. One is a prohibitively long runtime, the other is machine memory. Of these two issues, memory remains the limiting factor. The superlinear memory usage of current software means that even with improved hardware, the increase in what we are able to analyse will not increase substantially. 

One way to deal with the runtime is to use parallelization. 
In the software $\mathsf{Dory}$, parallelization was used to compute Vietoris-Rips persistent persistent homology for a point cloud with $10^{6}$ points and assisted by setting a small stopping radius for the filtration  
\cite{AGGARWAL2024102290}. 

The large memory requirements of currently available software stems from the fact that a large number of simplices are required for computation. 
Algorithms to reduce the number of simplices required for computation have been explored in \cite{glisse_et_al:LIPIcs.SoCG.2022.44}, though the author states that these techniques would be more suited for computing $\ph{q}(X)$ where $q\geq 2$ and that the results are mixed for computing $\ph{1}(X)$.  In \cite{koyama2024fastercomputationdegree1persistent}, the reduced Vietoris-Rips complex was introduced as a way to reduce the number of $2$-simplices required to compute $\ph{1}(X)$ from $O(n^3)$ to $O(n^2)$ provided $X$ has bounded doubling dimension. 

In this paper we present an algorithm that is highly parallelizable and memory efficient, though it may require a significant number of cores to obtain a runtime that is acceptable to the user.

\section{Background material}

Here we briefly present the necessary background material and establish notation. 

\subsection{Homology}

We give basic definitions for simplicial homology following Munkres' text \emph{Elements of Algebraic Topology}~\cite{munkres2018elements}.
First we  define the notion of a simplicial complex.

\begin{definition}[Simplicial complex]
    A simplicial complex with vertex set $V = \{v_{1},...,v_{n}\}$ is a set $K \subset 2^{V}$ which satisfies the following properties. 

    \begin{itemize}
        \item $\emptyset \in K$
        
        \item $\{v_{i}\} \in K$ for all $i\in \{1,...,n\}$. 

        \item If $\sigma \in K$, then all subsets of $\sigma$ are also in $K$. 

        \item If $\tau \in K$ and $\sigma \in K$ then $\tau \cap \sigma \in K$. 
    \end{itemize}
\end{definition}

Next we define the notion of a simplex. 

\begin{definition}[$q$-simplex]
    Consider a simplicial complex $K$ with vertex set $V = \{v_1,...,v_n\}$. Then let $\sigma$ be a subset of $V$ with $q+1$ elements. Then we refer to $\sigma$ as a $q$-simplex.
\end{definition}

We will make reference to the dimension of a simplex in Definition \ref{binary-relation-on-simplices}. 

\begin{definition}[Dimension of a simplex]
    Consider a simplicial complex $K$. Let $\sigma$ be a $q$-simplex, then we say that $\sigma$ has dimension $q$ and denote this by $\dim (\sigma) = q$
\end{definition}

It will be convenient later on to have a special term for when one simplex is a subset of another simplex. 

\begin{definition}[Faces and cofaces]
    Consider a simplicial complex $K$ with simplices $\sigma, \tau$ with $\sigma \subset \tau$. Then we say that $\sigma$ is a face of $\tau$ and $\tau$ is a coface of $\sigma$. 
\end{definition}

From here on in we write $\sigma = \{w_0,...,w_q\}$ as $\langle w_0...w_{q}\rangle$, following  standard notation for oriented simplices. Since we are working with $\mathbb{Z}_{2}$-coefficients we can effectively ignore the orientation of the simplices. This means that we can refer to a given simplex using any permutation of its vertices. For example, the $2$-simplex $\langle xyz \rangle$ can equally be referred to as $\langle xzy \rangle = \langle zxy \rangle = \langle zyx \rangle = \langle yzx \rangle = \langle yxz \rangle$.  
The addition of simplices is formalised in the next definition. 

\begin{definition}[Simplicial $q$-chain]
    A simplicial $q$-chain is a finite formal sum of $q$ simplices,  
    \begin{equation}
        \sum_{ i = 1}^{N}c_i\sigma_i
    \end{equation}
    In this paper, the coefficients $c_i$ are taken from $\mathbb{Z}_2$. 
\end{definition}

We now define three important vector spaces. 

\begin{definition}[Chain group]
    $C_{q}(K)$ is the free abelian group with coefficients in $\mathbb{Z}_2$ with generating set consisting of all $q$-simplices. It is customary to set $C_{-1}(K) = \mathbb{Z}_{2}$. 
\end{definition}

\begin{remark}
    It is worth noting that $C_{q}(K)$ is actually a vector space since $\mathbb{Z}_{2}$ is a field and that from this point on, any time the word ``group" is mentioned it could be replaced with ``vector space". We continue to use the word group to adhere to the ``traditional" presentation of homology, though the reader unfamiliar with groups can replace them with vector spaces for the purposes of this paper. 
\end{remark}

\begin{definition}[Boundary map]
    Let $K$ be a simplicial complex. Consider the map $\partial_{q+1}^{K}: C_{q+1}(K)\rightarrow C_{q}(K)$ defined as follows. Let $\sigma = \langle v_{0},...,v_{q}\rangle$. Then we define $\partial_{q+1}^{K} (\sigma)$ as follows:

    \begin{equation}
        \partial_{q+1}^{K} (\sigma) = \sum_{i=0}^{q}\langle v_{0}...\hat{v_i}...v_q \rangle  .
    \end{equation}
    We extend this linearly to a map on $C_{q+1}(K)$.
    Here $\hat{v_i}$ means that $v_i$ is to be omitted from $\langle v_0...v_i...v_q\rangle$. When it is clear what $K$ is, we may write $\partial_{q+1}^K$ as $\partial_{q+1}$. When $q$ is also clear, we may simply write $\partial_{q+1}$ as $\partial$. 
    The map $\partial_{0}^K: C_{0}(K) \rightarrow C_{-1}(K) = \mathbb{Z}_{2}$ acts in the following fashion:
    \begin{equation}
        \partial_{0}^{K} (\langle v_0 \rangle) = 1 .
    \end{equation}
It follows that $\partial_{0}^{K}$ maps a simplicial $0$-chain to the parity of the number of $0$-simplices in the chain. 
\end{definition}

\begin{definition}[$q$-cycles]
    Consider a simplicial complex $K$. We denote $\mathrm{ker}(\partial_q^K) = Z_{q}(K)$. We call a $q$-chain $c \in Z_{q}(K)$ a $q$-cycle. $Z_{q}(K)$ will be referred to as the group of $q$-cycles. 
\end{definition}

\begin{definition}[$q$-boundaries]
    Consider a simplicial complex $K$. We denote $\mathrm{im}(\partial_{q+1}) = B_{q}(K)$. We call a $q$-chain $c \in B_{q}(K)$ a $q$-boundary. $B_{q}(K)$ will be referred to as the group of $q$-boundaries. 
\end{definition}

We are now ready to define the homology groups of a simplicial complex $K$, but before we do, we state an extremely easy to verify lemma. 

\begin{lemma}
\label{boundary-of-boundary}
    Consider a simplicial complex $K$. Then we have $B_{q}(K) \subset Z_{q}(K)$. 
\end{lemma}

\begin{definition}[Homology groups of a simplicial complex]
Consider a simplicial complex $K$. Then the $q$th homology group is defined as $H_{q}(K) = Z_{q}(K)/B_{q}(K)$ 
\end{definition}

Note that Lemma $\ref{boundary-of-boundary}$ is necessary to show that $B_{q}(K)$ is indeed a subgroup of $Z_{q}(K)$ and hence the quotient group can be taken. It is here that we note that elements of $H_{q}(K)$ will be written as $[\gamma]$ to denote the fact that $\gamma$ is a representative of the class $[\gamma] \in H_{q}(K)$. Sometimes we will also use coset notation and write $[\gamma]$ as $\gamma + B_{q}(K)$. 

\subsection{Persistent Homology}

The following is a brief summary of some basic definitions in persistent homology. The reader who desires more context and details is directed towards chapter 7 of \emph{Computational Topology} \cite{book}, where the definitions below come from. 

\begin{definition}[Filtration of simplicial complexes]
Given an index set $\mathcal{I}$ and a set of simplicial complexes $(K_{i})_{i\in \mathcal{I}}$, if for $i \leq j$ in $\mathcal{I}$ we have $K_{i} \subset K_{j}$, we call the collection $(K_{i})_{i\in \mathcal{I}}$ a filtration of simplicial complexes.
\end{definition}

In the algorithm for computing persistent homology, we require 
a particular type of filtration.

\begin{definition}[Simplex-wise filtrations]
\label{def-simplex-wise-filtrations}
$(K_{i})_{i\in \mathcal{I}}$ is a simplex-wise filtration when $\mathcal{I} = \{0,...,m\} \subset \Z$ and $K_{i} = K_{i-1}\cup \sigma_{i}$  for $i \leq m$ where $\sigma_{i}$ is a single simplex and it is understood that $K_{0} = \emptyset$.

\end{definition}
As per this definition, $m$ will always denote the number of simplices of all possible dimensions in the filtration. 

Consider a simplex-wise filtration. For $i<j$ we apply the degree $q$ homology functor $H_{q}(-)$ to the inclusion $K_{i} \subset K_{j}$ to obtain a linear homomorphism $f_{i}^{j}: H_{q}(K_{i})\rightarrow H_{q}(K_j)$. Persistent homology quantifies how the homology changes across the parameter range. 

\begin{definition}[Birth index and giving birth]
Consider a simplex-wise filtration $(K_{i})_{i\in \{0,...,m\}}$. A homology class $[\gamma] \in H_{q}(K_{j})$ is said to be born at index $j$ if $j$ is the smallest index such that for all $j' < j$ there is no $[\beta] \in H_{q}(K_{j'})$ with  $f_{j'}^{j}([\beta ]) = [\gamma]$. The simplex $\sigma_j$ is said to give birth to the homology class $\gamma$. 
\end{definition}

\begin{definition}[Death index and killing]
\label{definition-of-death}
Consider a simplex-wise filtration $(K_{i})_{i\in \{0,...,m\}}$. A homology class $[\gamma]$ born at $j$ is said to die at index $k$ if $k$ is the smallest index such that $f_{j}^{k}([\gamma]) = f_{j'}^{k}([\beta])$ for some $j' < j$ and some $[\beta] \in H_{q}(K_{j'})$. The simplex $\sigma_{k}$, such that $K_{k} = K_{k-1}\cup \sigma_{k}$, is said to kill the homology class $[\gamma]$. 
\end{definition}

An important fact to establish is that for a simplex-wise filtration $(K_{i})_{i\in \{0,...,m\}}$ each simplex can only give birth to a homology class or kill a homology class. We prove this fact by first establishing some lemmas. 

\begin{lemma}
\label{can-only-birth}
    Let $(K_{i})_{i\in \{0,...,m\}}$ be a simplex-wise filtration and consider the change in homology between $K_{i-1}$ and $K_{i} = K_{i-1} \cup \sigma_{i}$, where $\sigma_{i}$ is a $q$-simplex. Then $\partial \sigma_{i} \in B_{q-1}(K_{i-1})$ if and only if $\sigma_{i}$ gives birth to a degree-$q$ homology class. 
\end{lemma}

\begin{proof}
    Suppose $\partial \sigma_{i} \in B_{q-1}(K_{i-1})$. Then we must have that $\partial \sigma_{i} = \sum_{j\in A} \partial \sigma_{j}$ where $A \subset \{0,...,i-1\}$. Then we have that $\partial \sigma_{i} - \sum_{j\in A}\partial \sigma_{j} = 0$ and thus $\partial (\sigma_{i} - \sum_{j \in A} \sigma_{j}) = 0$. Then we have that $\sigma_{i} - \sum_{j\in A}\tau_{j} + B_{q}(K_{i})$ is a homology class which is born at $i$. Note that $\sigma_{i} - \sum_{j\in A} \tau_{j} + B_{q}(K_{i})$ cannot possibly be  expressed in the form $f_{j}^{i}(\gamma + B_{q}(K_{j})) = \sigma_{i} - \sum_{j\in A}\tau_{j} + B_{q}(K_{i})$ since this would require $\gamma + \sum_{j \in A}\tau_{j} - \sigma_{i} \in B_{q-1}(K_{i})$, this cannot occur since all elements of $\gamma + \sum_{j \in A}\tau_{j}$ must be in $K_{i-1}$. Now we prove the converse, that is suppose that the addition of $\sigma_{i}$ entering the filtration gives birth to a degree-$q$ homology class $\gamma_1 + B_{q}(K_{i}) \in H_{q}(K_{i})$. We can show that $\gamma_{1}$ must be of the form $\sum_{k \in B}\sigma_{k} + \sigma_{i}$, where $B \subset \{0,...,i-1\}$. If this was not the case, i.e $\gamma_{1}$ was of the form $\sum_{k\in B}\sigma_{k}$ then we would have $\gamma_{1} + B_{q}(K_{\max(B)}) \in H_{q}(K_{\max (B)})$ meaning $f_{\max (B)}^{i}(\gamma_{1} + B_{q}(K_{\max(B)})) = \gamma_{1} + B_{q}(K_{i})$ contradicting the fact that $\gamma_1 + B_{q}(K_{i})$ was born upon the addition of $\sigma_{i}$. Since we know that $\partial (\gamma_{1}) = 0$ we have that $\partial (\sigma_{i} + \sum_{k \in B} \sigma_{k}) = 0$ which means that $ \partial \sigma_{i} = \sum_{k\in B}\partial \sigma_{k}$. Hence we have $\partial \sigma_{i} \in B_{q-1}(K_{i-1})$. 
\end{proof}

\begin{lemma}
\label{can-only-kill}
    Let $(K_{i})_{i\in \{0,...,m\}}$ be a simplex-wise filtration and consider the change in homology between $K_{i-1}$ and $K_{i} = K_{i-1} \cup \sigma_{i}$, where $\sigma_{i}$ is a $q$-simplex. Then $\partial \sigma_{i} \notin B_{q-1}(K_{i-1})$ if and only if $\sigma_{i}$ kills a degree-$(q-1)$ homology class. 
\end{lemma}

\begin{proof}
    First we will show that if the addition of $\sigma_{i}$ to the filtration kills a degree-$(q-1)$ homology class then we have $\partial \sigma_{i} \notin B_{q-1}(K_{i-1})$. To this end suppose that $\sigma_{i}$ kills a degree $(q-1)$ homology class $\gamma + B_{q-1}(K_{j})$ born at index $j$. Then that must mean we have some $\beta + B_{q-1}(K_{j'})$ with $j' < j$ such that $f_{j'}^{i}(\beta + B_{q-1}(K_{j'})) = f_{j}^{i}(\gamma + B_{q-1}(K_{j}))$. Then we have that $\beta + B_{q-1}(K_{i}) = \gamma + B_{q-1}(K_{i})$ and thus we have $\beta - \gamma \in B_{q-1}(K_{i})$. By Definition \ref{definition-of-death} we also know that $\beta - \gamma \notin B_{q-1}(K_{i-1})$. Now suppose $\partial \sigma_{i} \in B_{q-1}(K_{i-1})$, then we would have $B_{q-1}(K_{i-1}) = B_{q-1}(K_{i})$, implying $\beta -\gamma \in B_{q-1}(K_{i-1})$ a contadicition. Hence $\partial \sigma_{i} \notin B_{q-1}(K_{i-1})$. Now we prove the converse, suppose now that $\partial \sigma_{i} \notin B_{q-1}(K_{i-1})$. Since $\sigma_{i}$ is a $q$-simplex it follows that $Z_{q-1}(K_{i}) = Z_{q-1}(K_{i-1})$. Since all boundaries are cycles, we have $\partial \sigma_{i} \in Z_{q-1}(K_{i-1})$. We then have that $f_{0}^{i}(0+ B_{q-1}(K_{0})) = f_{i-1}^{i}(\partial \sigma_{i} + B_{q-1}(K_{i-1}))$ and $i$ is the lowest index such that this is the case since $f_{0}^{i-1}(0+ B_{q-1}(K_{0})) = 0 + B_{q-1}(K_{i-1}) \neq f_{i-1}^{i-1}(\partial \sigma_{i} + B_{q-1}(K_{i-1})) = \partial \sigma_{i} + B_{q-1}(K_{i-1})$
\end{proof}

\begin{lemma}
\label{can-only-birth-or-kill}
Let $(K_{i})_{i\in \{0,...,m\}}$ be a simplex-wise filtration and consider the change in homology between $K_{i-1}$ and $K_{i} = K_{i-1} \cup \sigma_{i}$, where $\sigma_i$ is a $q$-simplex.  Then one, and only one, of the following must occur. 

    \begin{itemize}
        \item $\sigma_i$ kills a homology class of degree $q-1$

        \item $\sigma_i$ gives birth to a homology class of degree $q$. 
    \end{itemize}
\end{lemma}

\begin{proof}
    Either $\partial \sigma_{i} \in B_{q-1}(K_{i-1})$ or $\partial \sigma_{i} \notin B_{q-1}(K_{i-1})$. Only one of these statements can be true and one of these statements must be true. By Lemma \ref{can-only-birth} the case  $\partial \sigma_{i} \in B_{q-1}(K_{i-1})$ corresponds to a birth of a degree-$q$ homology class and the case $\partial \sigma_{i} \notin B_{q-1}(K_{i-1})$ corresponds to a death of a degree $q-1$ homology class by Lemma \ref{can-only-kill}. 
\end{proof}

\begin{definition}[Persistence pair]
Consider a simplex-wise filtration $(K_{i})_{i\in \{0,...,m\}}$. If $[\gamma] \in H_{q}(K_{i})$ was born at index $i$ and died at index $j$ then $(i,j)$ is said to be a degree-$q$ persistence pair. We may also say that $(\sigma_{i}, \sigma_{j})$ is a degree-$q$ persistent pair when we want to use the simplicies themselves to denote the persistence pair rather than indices. If there is no confusion as to what the value of $q$ is, sometimes we will refer to $(i,j)$ simply as a persistence pair. 
\end{definition}

In section~\ref{section-distilled-vietoris-rips-complex}  we will make use a special type of persistence pair. The notion of apparent pairs was defined in \cite{Ripser}. They have also been referred to as ``close pairs'' in~\cite{Delgado-Friedrichs2015} and as ``shallow pairs" in \cite{edelsbrunner2024posetcancellationsfilteredcomplex}. 

\begin{definition}[Apparent pair]
    Consider a simplex-wise filtration $(K_{i})_{i\in \{0,...,m\}}$ Then a persistent pair $(\tau, \sigma)$ is called an apparent pair if the following two conditions both hold true. 

    \begin{itemize}
        \item Out of all faces of $\sigma$, $\tau$ is the face that appears the latest in the filtration. 

        \item Out of all cofaces of $\tau$, $\sigma$ is the face that appears earliest in the filtration. 
    \end{itemize}
    
\end{definition}

\subsection{Vietoris-Rips complexes and filtrations}

In this section we briefly discuss one of the main structures of interest for this paper. Before doing so, we define a measure of size for a simplex. 

\begin{definition}[Diameter of a simplex]
    Consider a finite point set $A \in \mathbb{R}^D$. Then the diameter $A$, denoted $\diam (A)$ is defined as $\max_{x,y\in A} d(x, y)$. For a simplex $\sigma = \langle x_{0}...x_{p}\rangle $, we have $\diam (\sigma) := \max_{i,j \in \{0,...,p\}}d(x_{i}, x_{j})$. 
\end{definition}

Vietoris-Rips complexes first appeared in \cite{vietoris1927hoheren} and were originally called Vietoris complexes. Vietoris-Rips complexes are also often referred to as Rips complexes in the TDA literature.

\begin{definition}[Vietoris-Rips complex at scale $r$]
Let $X$ be a finite metric space. For $r\in [0,\infty)$ we construct the Vietoris-Rips complex at scale $r$, $\vr _{r} (X)$, as follows. If $\{x_0,...,x_p\} \subset X$ is such that $\diam(\{x_0,...,x_p\}) \leq r$ then $\langle x_0...x_p \rangle $ is a $p$-simplex in $\vr_{r} (X)$. 
\end{definition}

\begin{remark}
    Throughout this paper we will frequently state that $X$ is to be a finite metric space, however we never require the triangle inequality. Thus one could replace the requirement that $X$ be a finite metric space to that $X$ be a finite semi-metric space. 
\end{remark}

In order to construct a filtration of simplicial complexes, we state an extremely easy to prove lemma without proof. 

\begin{lemma}
    \label{vr-is-actually-a-filtration}
    Let $X$ be a finite metric space and let $0 \leq r_1 \leq r_2$. Then we have $\vr_{r_1}(X) \subseteq \vr_{r_2}(X)$. 
\end{lemma}

\begin{definition}[Vietoris-Rips filtration]
    The Vietoris-Rips filtration on $X$ is the nested collection of spaces   $\vr _{\bullet} (X) := \{\vr_{r_1}(X)\subseteq \vr_{r_2}(X)\}_{0 \leq r_1 \leq r_2}$. 
\end{definition}

\begin{remark}
    $\vr_{\infty}(X)$ will consist of the $(|X|-1)$-simplex spanning  all points of $X$ and all lower-degree faces. It is the simplicial complex built from the power set (i.e., the set of all subsets, denoted $2^{X}$) of $X$. 
\end{remark}

The Vietoris-Rips filtration is not a simplex-wise filtration, but is easily modified to be so. We follow the method described in \cite{Ripser}. 
We first extend the function $\psi: X \rightarrow \{1,...,|X|\}$ which indexes the vertices, to a function that labels each simplex in $\vr_{\infty}(X)$. 

\begin{definition}[Extension of $\psi$]
\label{extention-of-psi}
    Given $\psi: X \rightarrow \{1,...,|X|\}$ we extend its domain and range to $\psi: \vr_{\infty}(X)\rightarrow 2^{\{1,...,|X|\}}$ in the following fashion. Consider a simplex $\sigma = \langle x_{0}...x_{p}\rangle$, then $\psi(\sigma)$ is the set of vertex labels $\{\psi(x_{0}),...,\psi(x_{p})\}$.
\end{definition}

Next, we define a function that sorts the integer labels in an element of $2^{\{1,...,|X|\}}$ so they are listed in increasing order. 

\begin{definition}
\label{sort-definition}
Let $A(n)$ be the set of \emph{ordered} subsets of $\{1,...,n\}$. 
That is, $S = \{s_1, \ldots, s_k\} \subset \{1,...,n\} $ is in $A(n)$ if and only if $s_1 < s_2 < \cdots < s_k$. 
We write $\sort(S)$ for the function that maps a set of integers to its ordered version. 
To shorten notation we will also write $\sort(\sigma)$ when we really mean $\sort(\psi(\sigma))$. 

\end{definition}

We now ``stretch out'' the Vietoris-Rips filtration and turn it into a simplex-wise filtration by using a length-lexicographic ordering on simplices with the same diameter. 

\begin{definition}
    \label{binary-relation-on-simplices}
    We define a binary relation $ <$ on $\vr_{\infty}(X)$ as follows. 

\begin{itemize}
    \item $\sigma <\tau$ if $\diam (\sigma) < \diam(\tau)$. 

    \item If $\diam(\sigma) = \diam(\tau)$ then $\sigma < \tau$ if $\dim (\sigma) < \dim (\tau)$. 

    \item If $\diam(\sigma) = \diam(\tau)$ and $\dim(\sigma) = \dim(\tau)$ then $\sigma < \tau$ if $\sort (\sigma) <_{lex} \sort(\tau)$ according to lexicographical order, $<_{lex}$. 
\end{itemize}
\end{definition}

Recall lexicographical ordering on elements of $A(n)$ with the same cardinality is defined as follows. 
Given $S, T \in A(n)$ we have $S = \{s_1 < s_2 \cdots < s_k\}$ and $T = \{t_1 < t_2 \cdots < t_k\}$. Then $S <_{lex} T$ in lexicographic ordering if there is some $1 \leq j \leq k$ such that $s_i = t_i$ for $i < j$, and $s_j < t_j$.   

The following lemma can be readily verified as length-lexicographic ordering is known to be a total order for finite sequences. The fact that $<$ defines a total order will be used to define what will be called the ``simplex-wise Vietoris-Rips filtration". 

\begin{lemma}
    The binary relation $<$ given in Definition \ref{binary-relation-on-simplices} is a total order on $\vr_{\infty}(X)$. 
\end{lemma}

\begin{definition}[Simplex-wise Vietoris-Rips filtration]
\label{def-VR-total-order}
Let $X$ be a finite metric space. Suppose there are $m$ simplices in $\vr _{\infty} (X)$. We use the total order $<$ on $\vr_{\infty}(X)$ to construct a bijection $\phi:\vr_{\infty}(X) \rightarrow \{1,...,m\}$ by mapping the lowest element according to $<$ to $1$, the next lowest element to $2$ and so on. The filtration $(K_{i})_{i \in \{0,...,m\}}$ where $K_{0} = \emptyset$ and $K_{i} = \cup_{j=1}^{i}\phi ^{-1}(j)$ for $i>0$ is referred to as the simplex-wise Vietoris-Rips filtration on $X$.
\end{definition}

\begin{remark}
    It is customary to write $\phi^{-1}(j)$ as $\sigma_j$. Thus we write $K_{i} = \cup_{j=1}^{i}\sigma_j$
\end{remark}

\begin{definition}[Birth and death values for  Vietoris-Rips filtrations]
Consider the collection $\vr_{\bullet}(X)$ and its corresponding simplex-wise filtration. Let $(i,j)$ be a persistence pair for the simplex-wise filtration. Then $\mathrm{diam}(\sigma_i)$ is said to be the birth value of the homology class that is born when $\sigma_i$ is added and $\mathrm{diam}(\sigma_j)$ is said to be the death value of this homology class. 
\end{definition}

\begin{definition}[Persistence]
Consider $\vr_{\bullet}(X)$ with its corresponding simplex-wise filtration. Let $(i,j)$ be a persistence pair for the simplex-wise filtration. Then the persistence of $(i,j)$ is defined as $\mathrm{diam}(\sigma_j) - \mathrm{diam}(\sigma_i)$.
\end{definition}

Note that the persistence may be zero. In this case we say that the persistence pair $(i,j)$ has trivial persistence.  

\begin{definition}[Persistence barcode, $\ph{q}(X)$]
Consider a finite metric space $X$. The multiset of all left-closed, right-open intervals $\left[\mathrm{diam}(\sigma_i), \mathrm{diam}(\sigma_j)\right)$ such that $(i,j)$ is a degree-$q$ persistence pair of the simplex-wise Vietoris-Rips filtration with non-trivial persistence is called the degree-$q$ Vietoris-Rips persistence barcode of $X$. We will denote it by $\ph{q}(X)$.

\end{definition}

The persistence barcode may be visualised by drawing each interval stacked above the real line as in Figure \ref{circle_of_circle_diag}. 
An alternative visualisation is the persistence diagram, where the points $(\mathrm{diam}(\sigma_i), \mathrm{diam}(\sigma_j))$ are plotted on cartesian axes for each persistence pair $(i,j)$.

\begin{figure}
\begin{subfigure}[t]{.5\textwidth}
    \centering
    \includegraphics[scale = 0.5, clip,trim=4cm 9cm 4cm 9cm]{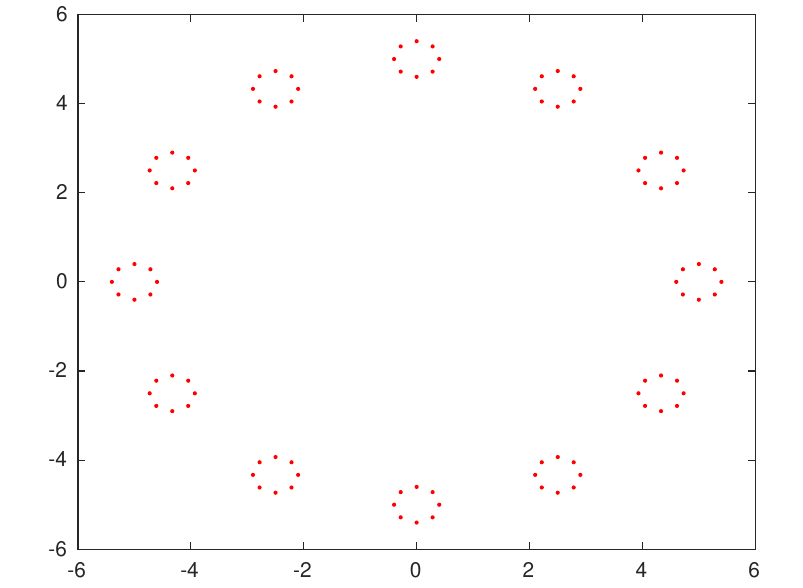}
    \subcaption{}
\end{subfigure}%
\begin{subfigure}[t]{.5\textwidth}
    \centering
\begin{tikzpicture}

\draw [|->] (0,0) -- (6.5,0);
\foreach \i in {0,2,4,6,8,10,12}
    \node [below] at (0.5*\i,0) {\i};

\foreach \i in {1,...,12}
    \draw [thick,red] (0.155,0.2*\i)--(0.4,0.2*\i);

\draw [thick, red] (0.9, 3.0) -- (4.6,3.0);
\end{tikzpicture}
 \subcaption{}
    \end{subfigure}
\caption{(a) A point cloud consisting of points on many small circles placed around a larger circle together with (b) its degree-1 persistence barcode. }
\label{circle_of_circle_diag}
\end{figure}
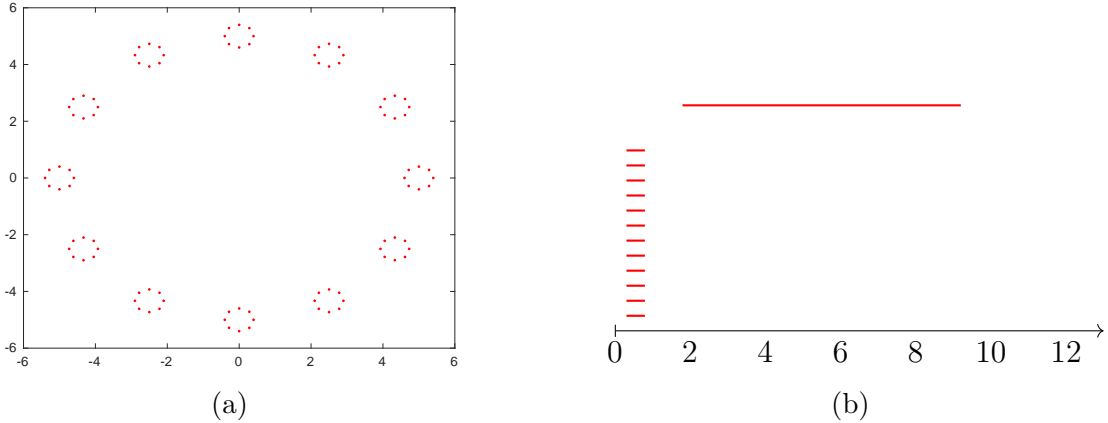

\subsection{Reduced Vietoris-Rips Filtrations}
In this section we recall several key concepts from \cite{koyama2024fastercomputationdegree1persistent}.

\begin{definition}[Lune]
Consider a finite metric space $X$ with its corresponding Vietoris-Rips filtration $\vr _{\bullet}(X)$. Then for a simplex $\sigma = \langle y_0....y_{p}\rangle$ we define $\lune (\sigma)$ in the following fashion. 

\begin{equation}
    \lune (\sigma) = \{x \in X \;|\; \langle xy_0...\hat{y_i}...y_{p} \rangle < \langle y_0...y_{p} \rangle \; \forall i\in \{0,...,p\} \}
\end{equation}

\end{definition}

We quantify structure in this subset by its connectivity. 

\begin{definition}[Connected components of a lune]
\label{def-connected-components-of-lune-appendix}
Consider a $q$-simplex $\sigma = \langle y_{0}...y_{q} \rangle$. Consider a graph with vertices consisting of the points in $\lune (\sigma)$. Join two points $p,q \in \lune ( \sigma )$ by an edge if $\langle pq y_{0}...\hat{y}_{i}...\hat{y}_{j}...y_{q} \rangle  < \sigma $ for all $0 \leq i < j \leq q$. Suppose this graph has $c$ connected components, then we say that $\lune (\sigma)$ has $c$ connected components. 
\end{definition}

We now describe how to build the reduced Vietoris-Rips complex. First we define a Lune function that records the selection of a single point per connected component of a lune. 

\begin{definition}[Degree-$q$ lune function]
\label{lune-function-appendix}
    Let $\vr^{q}_{\infty}(X)$ be the set of $q$-simplices in $\vr_{\infty}(X)$. We define a Lune function $L^{q}:\vr^{q}_{\infty}(X)\rightarrow 2^{X}$ as a function that takes a  $q$-simplex $\sigma$, with $c_{\sigma}$ connected components in its lune, to a set $\{x_{1},...,x_{c_{\sigma}}\}$. The points $x_{1},...,x_{c_{\sigma}}$ are chosen from the connected components of $\sigma$, one point from each connect component such that out of all points in a given component, they appear earliest in the filtration. 
\end{definition}

With the lune function defined, we are in a position to define degree-$q$ Reduced Vietoris-Rips complexes.

\begin{definition}
    \label{reduced-rips-higher-degree-homology}
    Consider a finite metric space $X$ and a lune function $L^{q}$. The degree-$q$ reduced Vietoris-Rips complex of $X$ with scale $r$, denoted $\rvr_{r}^{L^q}(X)$ is the simplicial complex consisting of 
    \begin{itemize}
        \item All $i$-simplices for $i= 0,1,...,q$

        \item ($q+1$)-simplices are as follows. 
    For each $q$-simplex $\sigma = \langle y_{0}...y_{q} \rangle \in \mathcal{R}^{L^q}_{r}(X)$ the $(q+1)$-simplices $\langle y_0...y_{q} x\rangle, x\in L^{q}(\sigma)$ are in $\mathcal{R}^{L^q}_{r}(X)$. 
    
    \end{itemize}
\end{definition}

Since $\rvr_r^{L^{q}}(X)$ is a simplicial complex and $\rvr^{L^{q}}_{r'}(X) \subset \rvr^{L^{q}}_{r}(X)$ for $r' <r$ it follows that we 
have a filtration for increasing scale $r$. 
\begin{definition} 
\label{def-reduced-vietoris-rips-filtration}
A degree-$q$ reduced Vietoris-Rips filtration on $X$ is any filtration of the form $\rvr^{L^{q}}_{\bullet}(X)$, where $L^{q}$ is a degree-$q$ lune function for $X$.
\end{definition}

We typically drop the $L^{q}$ from $\rvr_{\bullet}^{L^{q}}$ and write $\rvr_{\bullet}^{q}$ since it will be tacitly understood that there is a choice of degree-$q$ lune function. We now state a theorem from \cite{koyama2024fastercomputationdegree1persistent} which relates the persistent homology of the reduced Vietoris-Rips complex with the persistent homology of the standard Vietoris-Rips complex. 

\begin{theorem}
\label{deg-q-theorem-VR-RVR-isomorphism}
Consider a finite metric space $X$.   
Then there exists a family of isomorphisms $\theta^{q}_{\bullet}$ such that the following diagram commutes

\begin{equation}
 \begin{tikzcd}[ampersand replacement=\&]
  H_{q}(\rvr_{r_1}^{q}(X)) \arrow[r, "f_{r_1}^{r_2}"] \arrow[d, "\theta^{q}_{r_1}"'] \& H_{q}(\rvr_{r_2}^{q}(X)) \arrow[d, "\theta^{q}_{r_2}"] \\
  H_{q}(\vr_{r_1}(X)) \arrow[r, "g_{r_1}^{r_2}"] \& H_{q}(\vr_{r_2}(X))
  \end{tikzcd} 
\end{equation}

\noindent for all $r_1$ and $r_2$ such that $0 \leq  r_1 < r_2$.  Above, $f_{r_1}^{r_2}$ and $g_{r_1}^{r_2}$ are the maps at homology level induced by the natural inclusions. 
\end{theorem}

\subsection{Discrete Morse theory background}
\label{morse-theory-section}

In this section we will review the necessary theory for discrete morse Theory. Our presentation will follow Kozlov's \emph{Organized Collapse: An introduction to Discrete Morse Theory} \cite{kozlov2021organized} but will be much more terse, restricting to only what is absolutely necessary. The reader desiring a more complete treatment of discrete morse theory is naturally directed towards \cite{kozlov2021organized}. The user interested in a more classical presentation using discrete Morse functions is directed towards \cite{forman_DMT}.

\begin{definition}[Covering in a partially ordered set]
    Consider a partially ordered set $P$ with binary operator $``<"$. Then for two elements $\sigma, \tau \in P$ we say $\sigma$ covers $\tau$ if $\tau < \sigma$ and there exists no $\nu \in P$ such that $\tau < \nu < \sigma$. If $\sigma$ covers $\tau$ then we will write $\tau \prec \sigma$
\end{definition}

\begin{example}
    Consider a point cloud $X$. Then we can view $\vr_{\infty}(X)$ as a poset with $\tau < \sigma$ if $\tau$ is a face of $\sigma$. In this case $\sigma$ covers $\tau$ if $\sigma$ is a coface of $\tau$ of codimension $1$, i.e., $\dim (\sigma) = \dim (\tau) + 1$. 
\end{example}

\begin{example}
\label{poset-example}
    We can more generally consider a simplicial complex $K$. One can view $K$ as a poset with $\tau < \sigma$ if $\tau$ is a face of $\sigma$. In this case $\sigma$ covers $\tau$ if $\sigma$ is a coface of $\tau$ of codimension $1$, i.e., $\dim (\sigma) = \dim(\tau) + 1$. 
\end{example}

We are now ready to define the notion of a ``matching". Matchings are also commonly referred to as ``discrete vector fields" in the literature. Matchings can be defined in a more general context involving posets, but for our purposes we will always use a simplicial complex with its poset structure as discussed in Example $\ref{poset-example}$.

\begin{definition}
    Consider a simplicial complex $K$. Consider a subset $M \subset K$ and a bijection $\mu: M \rightarrow M$ that satisfies the following:

    \begin{itemize}
        \item For all $\sigma \in K$, $\mu(\sigma)$ covers $\sigma$ or vice-versa

        \item $\mu(\mu(\sigma)) = \sigma$
    \end{itemize}
    We refer to $\mu:M \rightarrow M$ as a ``partial matching" of $K$. 

\end{definition}

For our purposes, we will require that $\mu$ be acyclic. We will define the notion of an acyclic matching now. 

\begin{definition}
    Consider a poset $K$. Consider a partial matching $\mu: M \rightarrow M$ of $K$. Then $\mu$ is said to be an acyclic partial matching if there does not exist $\sigma_{1},\dots, \sigma_{l}$ such that:
    \begin{equation}
    \sigma_{1} \succ \mu(\sigma_{1}) \prec \sigma_{2} \succ \mu(\sigma_{2}) \prec \dots \prec \sigma_{l} \succ \mu(\sigma_{l}) \prec \sigma_{1} 
    \end{equation}
with $l \geq 2$ and the $\sigma_{i}$ distinct. 
\end{definition}

Given a partial matching $\mu$ of $K$, we can define subsets of $K$ using $\mu$. These subsets will be used later. 

\begin{definition}
\label{def:old-basis}
    Consider a poset $K$ with a partial acyclic matching $\mu: M \rightarrow M$ of $K$. Then we define the following subsets of $K$. 

    \begin{itemize}
        \item $M^{\uparrow}$ is the subset of all $\sigma \in M$, such that $\mu(\sigma)$ is covered by $\sigma$. 

        \item $M^{\downarrow}$ is the subset of all $\sigma \in M$, such that $\mu(\sigma)$ covers $\sigma$. 

        \item $R(\mu)$ is the complement of $M$. That is to say $R(\mu) := K\setminus M$.

    \end{itemize}

    If $K$ is a simplicial complex then we will also utilize the following notion. 

    \begin{itemize}
        \item $M_{q}^{\uparrow}$ is the subset of all $q$-simplices $\sigma \in M$, such that $\mu(\sigma)$ is covered by $\sigma$. 

        \item $M_{q}^{\downarrow}$ is the subset of all $q$-simplices $\sigma \in M$, such that $\mu(\sigma)$ covers $\sigma$. 

        \item $R_{q}(\mu)$ is the set of $q$-simplices in $M$ that aren't in $M_{q}^{\uparrow}\cup M_{q}^{\downarrow}$.

    \end{itemize}
\end{definition}

The simplices in $R_q(\mu)$ are often referred to as critical cells in the discrete Morse theory literature; they act as analogues of index-$q$ critical points in smooth Morse theory.

We will need the notion of a closure function in our discussion of discrete Morse theory. 

\begin{definition}
\label{definition-gqmu}
    Consider a simplicial complex $K$ with an acyclic partial matching $\mu$. For every $q$ between $0$ and $\dim (K)$, we define the directed graph $G_{q}(\mu)$ as follows. 
    \begin{itemize}
        \item The vertices of $G_{q}(\mu)$ are indexed by the $q$-dimensional simplices of $K$. 

        \item The edges of $G_{q}(\mu)$ are given by the rule: $(\sigma, \tau)$ is an edge of $G_{q}(\mu)$ if and only if $\mu(\tau)$ is defined, and $\sigma \succ \mu(\tau)$. 
    \end{itemize}

\end{definition}

\begin{remark}
    The acyclicity of $\mu$ implies the acyclicity of $G_{q}(\mu)$ for each $q = 0,1,...,\dim (K)$. 
\end{remark}

In \cite{kozlov2021organized}, it is stated that it is a well known fact in graph theory that the vertex set of a finite acyclic directed graph $G_{q}(\mu)$ can be decomposed into layers. That is, we can represent the graph as a disjoint union $V_{0}\cup V_{1} \cup V_{2} \cup ... \cup V_{t}$. 

\begin{itemize}
    \item For any $\beta \in V_{i-1}$, there exists $\alpha \in V_{i}$ such that $(\alpha, \beta)$ is an edge of $G_{q}(\mu)$. 

    \item For any $\alpha \in V_{j}, \beta \in V_{i}$, such that $(\alpha, \beta)$ is an edge of $G_{q}(\mu)$, we have $i < j$. 
\end{itemize}

The notion of a node in a graph being ``reachable" from another node will be useful in Section \ref{section-distilled-vietoris-rips-complex}.

\begin{definition}
    Consider a simplicial complex $K$ with an acyclic partial matching $\mu$. Then for a $q$-simplex $\sigma$, another $q$-simplex $\tau$ is said to be reachable from $\sigma$ if there exists a path connecting $\sigma$ to $\tau$ in $G_{q}(\mu)$. For a given $q$-simplex $\sigma$, we denote the set of $q$-simplices reachable from $\sigma$ in $G_{q}(\mu)$ by $\reach (\sigma)$. 
\end{definition}

We are now in a position to define the closure map $\varphi_{q}$ as a map that takes a $q$-simplex to a $\mathbb{Z}_{2}$ sum of $q$-simplices.

\begin{definition}[Closure map]
    $\varphi_{q}$ is defined recursively on the sets $V_{i}$ of $G_{q}(\mu)$ starting at $V_{0}$. To begin with, we will set $\varphi_{q}(\alpha) := \alpha$ for all $\alpha \in V_{0}$. 

    Suppose $\varphi_{q}$ has been defined on $V_{0},...,V_{i-1}$ for some $ 1\leq i \leq t$. Let $(\alpha, \beta_{1}), ..., (\alpha, \beta_{m})$ be the complete list of edges emanating from $\alpha \in V_i$. We then set 
    \begin{equation}
        \varphi_{q} (\alpha) := \alpha + \sum_{j=1}^{m}\varphi_{q} (\beta_{j}). 
    \end{equation}
\end{definition}

\begin{remark}
    For a given $q$-simplex $\sigma$, we have $\supp(\varphi_{q}(\sigma)) \subset \reach(\sigma)$. We may not have equality. This is because some $q$-simplices in $\reach(\sigma)$ may become cancelled out when computing $\varphi_{q}(\sigma)$.
\end{remark}

We now describe an alternative basis for $C_{q}(K)$. This new basis will be more convenient to use for later results. 

\begin{lemma}
\label{lemma: new-basis}
    Consider the sets

    \begin{itemize}
        \item $\mathcal{B}_{q}^{R} := \{\varphi_{q}(\gamma) \; | \; \gamma \in R_{q} \},$

        \item $\mathcal{B}_{q}^{\uparrow}:= M_{q}^{\uparrow},$

        \item $\mathcal{B}_{q}^{\downarrow}:= \{\partial \beta \; | \; \beta \in M_{q+1}^{\uparrow} \}$. 
    \end{itemize}

    Then the set $\mathcal{B}_{q}^{R}\cup \mathcal{B}_{q}^{\uparrow} \cup \mathcal{B}_{q}^{\downarrow}$ is a basis for $C_{q}(K)$. 
\end{lemma}

We introduce some more lemmas from \cite{kozlov2021organized}. This next lemma shows that $\partial$ maps $R_{q}(\mu)$ into $R_{q-1}(\mu)$. 

\begin{lemma}
\label{lemma:action-of-phi-on-R}
    For any $\alpha \in R_{q}$, there is a set $S(\alpha) \subset R_{q-1}$, such that 

    \begin{equation}
        \partial (\varphi_{q} (\alpha)) = \sum_{\beta \in S(\alpha)}\varphi_{q} (\beta)
    \end{equation}
\end{lemma}

We can use the new basis in order to split the chain complex $(C_{\bullet}(K), \partial)$ into a direct sum of chain complexes. 

\begin{theorem}
\label{thm:chain-complex-decomposition}
    Given a simplicial complex $K$ and an acyclic partial matching $\mu$, we define chain subcomplexes of $C_{*}(K)$, which we call $\crit _{*}(K,\mu)$ and $\match _{*}(K,\mu)$, as follows: for each $q \geq 0$, we take

    \begin{itemize}
        \item the group $\crit _{q}(K, \mu)$ to be generated by the set $\mathcal{B}_{q}^{R}$.

        \item the group $\match _{q}(K, \mu)$ to be generated by the set $\mathcal{B}_{q}^{\uparrow} \cup \mathcal{B}_{q}^{\downarrow}$.
    \end{itemize}

    The boundary operators on $\crit_{*}(K,\mu)$ and $\match _{*}(K,\mu)$ are simply $\partial_{q}$ restricted to $\crit_{q}(K,\mu)$ and $\match_{q}(K,\mu)$. The chain complex $C_{*}(K)$ decomposes as a direct sum 

    \begin{equation}
        C_{*}(K) = \crit_{*} (K, \mu) \oplus \match_{*} (K, \mu)
    \end{equation}
\end{theorem}

The proof of Theorem \ref{thm:chain-complex-decomposition} can be found in \cite{kozlov2021organized}, however we reproduce it here as well as it links Definition \ref{def:old-basis}, Lemma \ref{lemma: new-basis} and Lemma \ref{lemma:action-of-phi-on-R}.

\begin{proof}
    The decomposition $C_{q}(K) = \crit_{q}(K,\mu) \oplus \match_{q}(K,\mu)$ follows from the definitions. We need to show that the boundary operator is closed with respect to the chain sub-complexes $\crit_{*}(K,\mu)$ and $\match_{*}(K,\mu)$. Lemma \ref{lemma:action-of-phi-on-R} shows that $\partial$ is closed on $\crit_{*}(K,\mu)$. To see that $\partial$ is closed on $\match_{*} (K,\mu)$ observe that $\partial$ maps all elements of $\mathcal{B}_{q}^{\downarrow}$ to $0$ for all $q \geq 0$ and $\partial$ maps $\mathcal{B}_{q}^{\uparrow}$ to $\mathcal{B}_{q-1}^{\downarrow}$ for $q \geq 1$. For $q=0$, $\partial$ maps $\mathcal{B}_{0}^{\uparrow}$ to $0$. Thus we have shown that $\partial$ is closed with respect to the chain sub-complexes $\crit_{*}(K,\mu)$ and $\match_{*}(K,\mu)$. 
\end{proof}

In \cite{kozlov2021organized}, it is shown that $M^{\uparrow}$ does not contain any cycles. We make this notion precise in the following Lemma. 

\begin{lemma}
    \label{m-up-has-no-cycles}
    Consider a simplicial complex $K$ with a partial acyclic matching $\mu$. Suppose $\sigma \in C_{q}(K)$ and $\supp(\sigma) \subset M^{\uparrow}$. If $\partial \sigma = 0$ is a cycle then $\sigma = 0$
\end{lemma}

An important consequence of Lemma \ref{m-up-has-no-cycles} is that $\match_{*}(K,\mu)$ is acyclic. 

\begin{lemma}
    Consider a simplicial complex $K$ with a partial acyclic matching $\mu$. Then $\match_{*}(K, \mu)$ is acyclic. 
\end{lemma}

Using the fact that $\match_{*}(K,\mu)$ is acyclic we can relate the homology groups of $K$ to the homology groups of $\crit_{*}(K,\mu)$. 

\begin{theorem}
    \label{critical-is-everything}
    Given a simplicial complex $K$ and a partial acyclic matching $\mu$, we have the following for $q \geq 0$:

    \begin{equation}
        H_{q}(K) = H_{q}(\crit_{*}(K,\mu))
    \end{equation}
\end{theorem}

In \cite{kozlov2021organized}, it is shown that discrete Morse theory can be applied to a filtration of topological spaces $(K_{i})_{i\in I}$ where $I$ is a totally ordered set and $K_{\infty}:= \cup_{i\in I}K_{i}$ is a finite simplicial complex. We start with some preliminary definitions.

\begin{definition}
    Consider a filtration of finite simplicial complexes $(K_{i})_{i \in I}$. Let $\sigma$ be a simplex, then we define $h(\sigma)$ as the smallest value of $i$ such that $\sigma \in K_{i}$. 
\end{definition}

\begin{remark}
    In the context of Vietoris-Rips persistent homology, for a given simplex $\sigma$, $h(\sigma) = \diam(\sigma)$ as $\diam(\sigma)$ is the smallest value of $r$ such that $\sigma \in \vr_{r}(X)$. 
\end{remark}

In order to apply discrete Morse theory to persistent homology, we require $\mu$ to respect the filtration. 

\begin{definition}
    Consider a partial acyclic matching $\mu$ on $K_{\infty} := \cup_{i\in I}K_{i}$. Then $\mu$ is said to respect the filtration $K_{i}$ if $h(\mu(\sigma)) = h(\sigma)$ for all $\sigma \in K_{\infty}$ such that $\mu(\sigma)$ is defined. 
\end{definition}

From the filtration $(K_{i})_{i\in I}$ we can construct an analogous chain complex.

\begin{definition}
    Consider a filtration $(K_{i})_{i\in I}$ and $q > 0$ and a partial acyclic matching $\mu$ on $K_{\infty}$. Then for each $i \in I$ we define a chain complex $\crit_{*}(K_{\infty}, \mu)_{i}$ by defining $\crit_{q}(K_{\infty}, \mu)_{i}:= \crit_{q}(K_{i}, \mu)$. 
\end{definition}

We are now ready to give a theorem relating the degree-$q$ persistent homology of $(K_{i})_{i\in I}$ to the chain complexes $(\crit_{*}(K_{i}, \mu))_{i\in I}$. 

\begin{theorem}
\label{morse-crit-isomorphism}
    Consider a filtration of finite simplicial complexes $(K_{i})_{i\in I}$ with $i$ a totally ordered set. Then we have a family of isomorphisms $\theta_{\bullet}^{q}$ such that the following diagram commutes for all $i < j$ in $I$ and all $q \geq 0$. 

\begin{equation}
 \begin{tikzcd}[ampersand replacement=\&]
  H_{q}(\crit_{q}(K_{i}, \mu)) \arrow[r, "f_{i}^{j}"] \arrow[d, "\theta^{q}_{i}"'] \& H_{q}(\crit_{q}(K_{j}, \mu)) \arrow[d, "\theta^{q}_{j}"] \\
  H_{q}(K_{i}) \arrow[r, "g_{i}^{j}"] \& H_{q}(K_{j})
  \end{tikzcd} 
\end{equation}

Here $f_{i}^{j}$ is the map at homology level induced by the inclusion $\crit_{q}(K_{i},\mu) < \crit_{q}(K_{j},\mu)$ as groups. $g_{i}^{j}$ is the map at homology level induced by the inclusion $K_{i} \subset K_{j}$. The map $\theta_{i}^{q}$ is the map at homology level induced by sending $\varphi_{q}(\sigma)$ in $\crit_{q}(K_{i},\mu)$ to $C_{q}(K_{i})$, $\theta_{j}^{q}$ is defined analogously. 

\end{theorem}

\section{The Distilled Vietoris-Rips Complex}
\label{section-distilled-vietoris-rips-complex}

In this section we will apply the Discrete Morse Theory from Section \ref{morse-theory-section} to the degree-$q$ Reduced Vietoris-Rips complex. We begin by stating a fact about apparent pairs from \cite{Ripser}. 

\begin{lemma}
   \label{apparent-pair-lemma}
    The apparent pairs of the simplex-wise refinement of the Vietoris-Rips filtration form an acyclic partial matching
\end{lemma}

For a given $q$-simplex $\sigma$ with non-empty lune, there is a way to construct the apparent pair it belongs to. 

\begin{lemma}
    \label{lune-apparent-pair}
    Consider a metric space $X$. Let $\tau = \langle y_{0}...y_{q} \rangle$ be a $q$-simplex in $\vr_{\infty}(X)$. Suppose $\lune(\tau) \neq \emptyset$ and let $ x \in \lune(\tau)$ be such that $\langle x \rangle$ is the $0$-simplex that appears earliest in the simplex-wise Vietoris-Rips filtration. Then $(\tau, \langle x\tau \rangle)$ is an apparent pair. 
\end{lemma}

\begin{proof}
    We first show that $\langle x \tau \rangle$ is the $(q+1)$-simplex containing $\tau$ that appears earliest in the filtration. Suppose for contradiction there was another $(q+1)$-simplex $\langle z \tau \rangle$ that appeared in the filtration before $\langle x \tau \rangle$. Note that since $\tau < \langle z \tau \rangle < \langle x \tau \rangle$ in the filtration it follows that $\diam (\tau) \leq \diam (\langle z \tau \rangle )\leq \diam (\langle x \tau \rangle)$. Since $x\in \lune(\tau)$ we have $\diam(\tau) = \diam(\langle x \tau \rangle)$ and thus it follows that $\diam(\langle z \tau \rangle) = \diam(\tau) = \diam(\langle x \tau \rangle)$. Since $\langle z \tau \rangle$ and $\langle x \tau \rangle$ are both $(q+1)$-simplices, have the same diameter and $\langle z \tau \rangle < \langle x \tau \rangle$  it follows that the $0$-simplex $\langle z \rangle$ appeared in the filtration before the $0$-simplex $\langle x \rangle$. Since $x\in \lune(\tau)$ it follows that $\langle x y_0...\hat{y_i}...y_{q} \rangle < \langle y_0...y_{p} \rangle \; \forall i\in \{0,...,q\}$. We now show that $\langle z y_0...\hat{y_i}...y_{q} \rangle < \langle y_0...y_{q} \rangle \; \forall i\in \{0,...,q\}$. Suppose for contradiction there existed a particular $i$ such that $\langle y_0...y_{q} \rangle <  \langle zy_0...\hat{y_i}...y_{q} \rangle$. Then since we have $\langle y_0...y_{q} \rangle <  \langle zy_0...\hat{y_i}...y_{q} \rangle < \langle z\tau \rangle$ it follows that $\diam (\langle zy_0...\hat{y_i}...y_{q} \rangle) = \diam(\tau)$. It thus follows that $\sort (\langle zy_0...\hat{y_i}...y_{q} \rangle) $ appears lexicographically after $\sort( \langle y_0...y_{q} \rangle)$. Thus it follows that $\langle z \rangle$ appears in the filtration before $\langle y_{i} \rangle$. But we also have $\langle x y_0...\hat{y_i}...y_{q} \rangle < \langle y_0...y_{q} \rangle$ which implies that $\langle x \rangle  < \langle y_{i} \rangle$. Thus we have $\langle x \rangle < \langle y_{i} \rangle < \langle z \rangle$, a contradiction. The fact that $\tau$ is the face of $\langle x \tau \rangle$ that appears latest in the filtration is a direct consequence of the fact that $x\in \lune(\tau)$ and the definition of $\lune(\tau)$. 
\end{proof}

Using the fact that the apparent pairs form an acyclic partial matching on $\vr_{\infty}(X)$ we can construct a matching on $\rvr_{\infty}^{q}(X)$.

\begin{definition}[Partial matching on the degree-$q$ reduced Vietoris-Rips complex]
    \label{deg-q-reduced-vietoris-rips-complex-matching}
    Consider the degree-$q$ Reduced Vietoris-Rips complex $\rvr_{\bullet}^{q}(X)$. We will construct a matching $\mu_{q}$ as follows: For each $q$-simplex $\sigma = \langle y_{0} ... y_{q} \rangle $ such that $L^{q}(\sigma) \neq \emptyset$ choose $x_{\sigma} \in L^{q}(\sigma)$ such that $\langle x_{\sigma} \rangle$ is the $0$-simplex which appears earliest in the filtration. Then set $\mu_{q}(\sigma) = \langle x_{\sigma} \sigma \rangle$. 
\end{definition}

Note that Lemmas \ref{apparent-pair-lemma} and \ref{lune-apparent-pair} imply that the matching given in Definition \ref{deg-q-reduced-vietoris-rips-complex-matching} is indeed acyclic. We also need to show that the matching given in Definition \ref{deg-q-reduced-vietoris-rips-complex-matching} respects the Vietoris-Rips filtration. 

\begin{lemma}
    The matching $\mu_{q}$ given in Definition \ref{deg-q-reduced-vietoris-rips-complex-matching} respects the Vietoris-Rips filtration. 
\end{lemma}

\begin{proof}
    For any $q$-simplex $\sigma$ with $\lune(\sigma) \neq \emptyset$ and $q \geq 1$ we have $\diam (\sigma) = \diam (\mu_{q} (\sigma))$
\end{proof}

We are now in a position to define the degree-$q$ Distilled Vietoris-Rips Complex at scale $r$. 

\begin{definition}[Degree-$q$ distilled Vietoris-Rips complex at scale $r$]
    \label{deg-q-distilled-vietoris-rips-complex-at-scale-r}
    Consider a metric space $X$. Consider the partial matching $\mu_{q}$ defined on the degree-$q$ Reduced Vietoris-Rips complex as defined in Definition \ref{deg-q-reduced-vietoris-rips-complex-matching}. Let $A$ be given by the following set 
    \begin{equation}
        A :=  \bigcup_{\{\sigma \in R_{q+1}(\mu_{q}) \}}\reach(\sigma) 
    \end{equation}
\noindent
We then define $\dvr_{r}^{q}(X)$ as follows
    \begin{equation}
        \dvr_{r}^{q}(X) := \{\nu \in A \; | \; \diam (\nu) \leq r\} \cup \{\eta \subset \nu \; | \; \nu \in A, \diam (\nu) \leq r \}
    \end{equation}

\end{definition}

It readily follows from Definition \ref{deg-q-distilled-vietoris-rips-complex-at-scale-r} that for $r_{1} < r_{2}$ we have $\dvr _{r_1}^{q} (X) \subset \dvr_{r_2}^{q}(X)$.

\begin{definition}
    We denote $\dvr_{\bullet}^{q}(X)$ as the Distilled Vietoris-Rips filtration. 
\end{definition}

We wish to apply Discrete Morse Theory to the filtration $\dvr_{\bullet}^{q}(X)$. In order to do so, we need to define a partial acyclic matching on $\dvr_{\bullet}^{q}(X)$. We do so by simply restricting $\mu_{q}$ to $\dvr_{\bullet}^{q}(X)$. The reader may be skeptical of whether restricting $\mu_{q}$ to $\dvr_{\bullet}^{q}(X)$ still respects the filtration $\dvr_{\bullet}^{q}(X)$. The next lemma shows that this is indeed the case, at least for the simplices required to compute degree-$q$ persistent homology. 

\begin{lemma}
\label{matching-is-in-simplicial-complex}
    Let $r >0$. Suppose $\alpha \in \dvr_{r}^{q}(X)$ is a $q$-simplex such that $\mu_{q}(\alpha) \in \rvr_{r}^{q}(X)$ is defined. Then $\mu_{q}(\alpha) \in \dvr_{r}^{q}(X)$.  
\end{lemma}

\begin{proof}
    From the definition of $\dvr_{r}^{q}(x)$ we have that $\alpha$ is a face of a $(q+1)$-simplex $\sigma$ with $\diam(\sigma) \leq r$ that does appear in $A$. The fact that $\sigma \in \dvr_{r}^{q}(X)$ means that the simplices in $\reach(\sigma)$ appear in $\dvr_{r}^{q}(X)$. $\mu_{q}(\mu_{q}(\alpha)) = \alpha$ is certainly defined and $\mu_{q}(\mu_{q}(\alpha)) \prec \sigma$. Thus by Definition \ref{definition-gqmu} we have that $(\sigma, \mu_{q}(\alpha))$ is an edge in $G_{q+1}(\mu_{q})$.  It then follows from the definition of $\reach(\sigma)$ that $\mu_{q}(\alpha)\in \reach(\sigma)$.     
\end{proof}

Before proceeding to the proof of the main result, we prove a useful Lemma. 

\begin{lemma}
\label{lemma-same-homology-red-dist}
    Consider a finite metric space $X$ and $q >0$. Then we have

    \begin{equation}
   H_{q}(\crit_{*}(\dvr_{r}^{q}(X), \mu_{q})) = H_{q}(\crit_{*}(\rvr_{r}^{q}(X), \mu_{q})) 
    \end{equation}

This is a genuine equality and not merely an isomorphism. 
\end{lemma}

\begin{proof}
    It follows directly from the definition of $\dvr_{r}^{q}(X)$ that the critical $(q+1)$ simplices are the same. For each critical $(q+1)$-simplex $\sigma$, $\reach(\sigma)$ will be the same for both $\dvr_{r}^{q}(X)$ and $\rvr_{r}^{q}(X)$, so we have that $\crit_{q+1}(\dvr_{r}^{q}(X), \mu_{q}) = \crit_{q+1}(\rvr_{r}^{q}(X), \mu_{q})$ which implies $B_{q}(\crit_{*}(\dvr_{r}^{q}(X), \mu_{q})) = B_{q}(\crit_{*}(\rvr_{r}^{q}(X), \mu_{q}))$. Fix $r > 0$, we will show that $Z_{q}(\crit_{*}(\dvr_{r}^{q}(X))) = Z_{q}(\crit_{*}(\rvr_{r}^{q}(X)))$. We trivially have $Z_{q}(\crit_{*}(\dvr_{r}^{q}(X))) \subset Z_{q}(\crit_{*}(\rvr_{r}^{q}(X)))$, so we now prove the reverse inclusion. Consider a cycle $\gamma$ in $\crit_{q}(\rvr_{r}^{q}(X), \mu_{q})$, since \emph{all} degree-$q$ homology classes eventually die we know that for some $t > 0$ that $\gamma \in B_{q}(\crit_{*}(\rvr_{t}^{q}(X)), \mu_{q}) = \partial (\crit_{q+1}(\rvr_{t}^{q}(X)), \mu_{q}) = \partial (\crit_{q+1}(\dvr_{t}^{q}(X)), \mu_{q})$. Since all elements of $\partial (\crit_{q+1}(\dvr_{t}^{q}(X)), \mu_{q})$ appear in $\dvr_{\bullet}^{q}(X)$ it follows that $\gamma$ is a cycle in  $\crit_{q}(\dvr_{r}^{q}(X), \mu_{q})$. Thus we have $Z_{q}(\crit_{*}(\dvr_{r}^{q}(X)), \mu_{q}) = Z_{q}(\crit_{*}(\rvr_{r}^{q}(X)),\mu_{q})$ and the proof is complete. 
\end{proof}

Now we show that the degree-$q$ persistent homology derived from the Distilled Vietoris-Rips filtration is the same as the degree-$1$ persistent homology of the Reduced Vietoris-Rips filtration which is the same as the degree-$1$ persistent homology of the standard Vietoris-Rips persistent homology. 

\begin{theorem}
    \label{main-theorem-of-paper}
    Consider a finite metric space $X$ and $q >0$. Then there exists a family of isomorphisms $\psi^{q}_{\bullet}$ such that the following diagram commutes. 
    \begin{equation}
 \begin{tikzcd}[ampersand replacement=\&]
  H_{q}(\rvr_{r_1}^{q}(X)) \arrow[r, "f_{r_1}^{r_2}"] \arrow[d, "\psi_{r_1}"'] \& H_{q}(\rvr_{r_2}^{q}((X)) \arrow[d, "\psi_{r_2}"] \\
  H_{q}(\dvr_{r_1}^{q}(X)) \arrow[r, "g_{r_1}^{r_2}"] \& H_{q}(\dvr_{r_2}^{q}(X))
  \end{tikzcd} 
\end{equation}
Here $f_{r_1}^{r_2}$ and $g_{r_1}^{r_2}$ are the maps obtained by applying $H_{q}(-)$ to the inclusions $\rvr_{r_1}^{q}(X) \subset \rvr_{r_2}^{q}(X)$ and $\dvr_{r_1}^{q}(X) \subset \dvr_{r_1}^{q}(X)$. 
\end{theorem}

\begin{proof}
   We know from Theorem \ref{morse-crit-isomorphism} that we have the following commutative diagrams.   
\begin{equation}
\label{cd-1}
 \begin{tikzcd}[ampersand replacement=\&]
  H_{q}(\crit_{*}(\rvr_{r_1}^{q}(X), \mu_{q})) \arrow[r, "f_{r_1}^{r_2}"] \arrow[d, "\theta^{q}_{r_1}"'] \& H_{q}(\crit_{*}(\rvr_{r_2}^{q}(X), \mu_{q})) \arrow[d, "\theta^{q}_{r_2}"] \\
  H_{q}(\rvr_{r_1}^{q}(X)) \arrow[r, "g_{r_1}^{r_2}"] \& H_{q}(\rvr_{r_2}^{q}(X))
  \end{tikzcd} 
\end{equation}
Here, $f_{r_1}^{r_2}$ and $g_{r_1}^{r_2}$ are the maps at the homology level induced by the inclusions $\crit_{q}(\rvr_{r_1}^{q}(X), \mu) < \crit_{q}(\rvr_{r_2}^{q}(X), \mu)$ and $\rvr_{r_1}^{q}(X) \subset \rvr_{r_2}^{q}(X)$, $\theta_{r_1}^{q}$ and $\theta_{r_2}^{q}$ are isomorphisms. 

\begin{equation}
\label{cd-2}
 \begin{tikzcd}[ampersand replacement=\&]
 H_{q}(\dvr_{r_1}^{q}(X)) \arrow[r, "j_{r_1}^{r_2}"] \& H_{q}(\dvr_{r_2}^{q}(X))\\
  H_{q}(\crit_{*}(\dvr_{r_1}^{q}(X), \mu_{q})) \arrow[r, "h_{r_1}^{r_2}"] \arrow[u, "\alpha^{q}_{r_1}"'] \& H_{q}(\crit_{*}(\dvr_{r_2}^{q}(X), \mu_{q})) \arrow[u, "\alpha^{q}_{r_2}"]
  \end{tikzcd} 
\end{equation}
Here, $h_{r_1}^{r_2}$ and $j_{r_1}^{r_2}$ are the maps at the homology level induced by the inclusions $\crit_{q}(\dvr_{r_1}^{q}(X), \mu) < \crit_{q}(\dvr_{r_2}^{q}(X), \mu)$ and $\dvr_{r_1}^{q}(X) \subset \dvr_{r_2}^{q}(X)$, $\alpha_{r_1}^{q}$ and $\alpha_{r_2}^{q}$ are isomorphisms. Using Lemma \ref{lemma-same-homology-red-dist} we can combine the commutative diagrams in (\ref{cd-1}) and (\ref{cd-2}) to get 

\begin{equation}
\label{cd-3}
 \begin{tikzcd}[ampersand replacement=\&]
  H_{q}(\dvr_{r_1}^{q}(X)) \arrow[r, "j_{r_1}^{r_2}"] \& H_{q}(\dvr_{r_2}^{q}(X))\\
  H_{q}(\crit_{*}(\dvr_{r_1}^{q}(X), \mu_{q})) \arrow[r, "h_{r_1}^{r_2}"] \arrow[u, "\alpha^{q}_{r_1}"'] \arrow[d,equal] \& H_{q}(\crit_{*}(\dvr_{r_2}^{q}(X), \mu_{q})) \arrow[u, "\alpha^{q}_{r_2}"] \arrow[d,equal]\\
 H_{q}(\crit_{*}(\rvr_{r_1}^{q}(X), \mu_{q})) \arrow[r, "f_{r_1}^{r_2}"] \arrow[d, "\theta^{q}_{r_1}"'] \& H_{q}(\crit_{*}(\rvr_{r_2}^{q}(X), \mu_{q})) \arrow[d, "\theta^{q}_{r_2}"] \\
  H_{q}(\rvr_{r_1}^{q}(X)) \arrow[r, "g_{r_1}^{r_2}"]\& H_{q}(\rvr_{r_2}^{q}(X))
  \end{tikzcd} 
\end{equation}

We can define $\psi_{r}^{q} = \alpha_{r}^{q} \circ (\theta_{r}^{q})^{-1}$ and thus the proof is complete. 

\end{proof}

\section{A highly parallelizable algorithm for computing \\ Vietoris-Rips persistent homology in degree-1}

In this section we use the Distilled Vietoris-Rips complex to create a highly parrallelizable algorithm for computing $\ph{1}(X)$. This algorithm first computes the filtration $\dvr_{\bullet}^{1}(X)$ and then applies standard reduction methods to the Distilled Vietoris-Rips complex. For this section, we assume $X$ is a finite metric space that has finite doubling dimension. 
Requiring finite doubling dimension ensures that the number of connected components in each lune is bounded \cite{koyama2024fastercomputationdegree1persistent}. Algorithm \ref{main-algorithm} can be extended to higher degree homologies, but will likely be too slow for practical purposes. The parallelizable part of Algorithm \ref{main-algorithm} is in the main for-loop. 

\begin{algorithm}
\caption{Algorithm to compute $\ph{1}(X)$ using $\dvr^{1}_{\bullet}(X)$}
\label{main-algorithm}
\KwData{Distance matrix $d_{X}$}
\KwResult{$\ph{1}(X)$}

 \For{$e \in \vr^{1}_{\infty}(X)$}{
    $n_e \longleftarrow$ the number of connected components of $\lune (e)$ \;
    \If{$n_e > 1$}{
        Choose one point from each of the $n_e$ connected components. Call these points $x_{1},...,x_{n_e}$. \;
        \For{$i = 2,...,n_e$}{
            Add $\reach(\langle x_{i} e \rangle)$ to $\dvr^{1}_{\bullet}(X)$. \;
            If not already in $\dvr^{1}_{\bullet}(X)$, add all faces of $\reach(\langle x_{i} e \rangle)$ to $\dvr^{1}_{\bullet}(X)$. 
        }
    }
 }

 \tcc{At this stage of the algorithm, you have $\dvr^{1}_{\bullet}(X)$}

 Compute $\ph{1}(X)$ using $\dvr^{1}_{\bullet}(X)$ \;
 
\end{algorithm}

\subsection{Complexity Analysis}
    Here we provide a complexity analysis for Algorithm \ref{main-algorithm}. The main for-loop will be run $O(n^2)$ times. For a given $e\in \mathcal{V}_{\infty}^{1}(X)$ we compute $\lune(e)$ which will have complexity $O(n)$. Computing the number of connected components of $\lune(e)$ will have a complexity of $O(n^2\log (n))$. The number of connected components is computed by first computing the MST on the points in $\lune(e)$, which has time complexity $O(n^2\log (n))$. The MST constructed on $\lune(e)$ will have $O(n)$ edges, and as a result performing a union find algorithm using this MST will have complexity $O(n\alpha (n))$, where $\alpha$ is the inverse Ackermann function. Thus the complexity of finding the number of connected components of $\lune (e)$ is dominated by $O(n^2\log(n))$. We then need the complexity of finding $\reach(\langle x_{i}e \rangle)$. This is rather difficult to estimate since we do not know the size of $\reach(\langle x_{i} e \rangle))$ a-priori to computing it. We can however provide a very crude upper bound. 
    
    Since we constructed $\mathcal{D}^{1}_{\bullet}(X)$ from $\rvr^{1}_{\bullet}(X)$ we know that the size of $\reach(\langle x_{i} e \rangle)$ is bounded above by $O(n^2)$, since there are that many $2$-simplices in $\rvr^{1}_{\bullet}(X)$. For a given $2$-simplex $\alpha$ (not necessarily in $R_{2}(\mu_{1})$), consider what needs to be computed in order to find $\beta_{1}, \beta_{2},...,\beta_{m}$ such that $(\alpha, \beta_{1}), ..., (\alpha, \beta_{m})$ is the complete list of edges emanating from $\alpha$ in $G_{2}(\mu_1)$. If $\alpha$ is in $R_{2}(\mu_1)$ then $\mu_1(\alpha)$ is not defined. Let $u(\alpha)$ be the face of $\alpha$ which appears in the filtration latest. Then, using the same notation as in Theorem \ref{main-theorem-of-paper} let $l(\alpha)$ and $r(\alpha)$ be the other two faces of $\alpha$. Since we do not explicitly store the matching $\mu_1$ we need to find $\mu_1(u(\alpha)), \mu_1(l(\alpha))$ and $\mu_1(r(\alpha))$ by computing the lunes of $u(\alpha), l(\alpha)$ and $r(\alpha)$. Note that the lunes of $l(\alpha)$ and $r(\alpha)$ may not exist since $r(\alpha)$ and $l(\alpha)$ may be in $R_{1}(\mu_{1})$. $u(\alpha)$ will not have an empty lune since the $0$-simplex face common to $l(\alpha)$ and $r(\alpha)$ will be, by definition, contained in $\lune (u(\alpha))$. We have already shown, for a given $1$-simplex $e$, the complexity of finding $\lune (e)$ and the number of connected components of $\lune(e)$ is bounded by $O(n^2 \log (n))$. If $\alpha$ is not in $R_{2}(\mu_1)$ then the same argument applies with $u(\alpha)$ replaced with $\mu(\alpha)$. Thus the process of computing $\reach (\alpha)$ consists of computing the lune and finding its connected components for up to $O(n^2)$ times. Thus the complexity of computing $\reach (\langle x_{i} e \rangle)$ can be bounded by $O(n^4 \log (n))$. Since the main for loop is executed $O(n^2)$ times it follows that the total complexity can be bounded by $O(n^6 \log (n))$. 

    It should be said that $O(n^6 \log (n))$ is an extremely conservative overestimate. For the overwhelming majority of $e\in \vr_{\infty}^{1}(X)$, $\lune(e)$ will have one connected component. This means that with the exception of a small subset of $\vr_{\infty}^{1}(X)$, the code inside the main for-loop witll have complexity $O(n)$. Let $b(X)$ denote the number of $1$-simplices with more than $1$-connected component in their lune. Then the complexity can be rewritten in the form 

    \begin{equation}
        O( (n^2 - b(X))n + b(X)n^4 \log (n))
    \end{equation}

    Typically $b(X)$ is of the order $O(n)$ so this gives a slightly more favourable complexity of $O(n^5 \log (n))$. 
    
    The benefit of this algorithm is that it is highly parralelizable. We now compute the complexity of the algorithm when the main loop is run over $m$ machines. We assumme that each machine has access to the distance matrix and $\dvr_{\bullet}^{1}(X)$. Since we are dividing the execution of the main loop over $m$ machines the overall complexity is divided by $m$. That is, the complexity is reduced to 

    \begin{equation}
        O\left( \tfrac{1}{m}(n^2 - b(X))n + b(X)n^4 \log (n) \right)
    \end{equation}

    In practice, the time complexity will be larger than this, as there will be some time complexity associated with relaying the results to the final machine that compute the persistent homology. 

    \subsection{Memory Usage}
    One of the benefits of Algorithm \ref{main-algorithm} is that it will have low memory usage. Until one actually computes $\ph{1}(X)$ using $\dvr_{\bullet}^{1}(X)$ the only things that need to be stored are the $2$-simplices that will be put into $\dvr_{\bullet}^{1}(X)$ (since we can always obtain the $1$-simplices by simply looking at the faces of the $2$-simplices) and $d_{X}$. Thus, excluding the storage of $d_{X}$, the memory usage of Algorithm \ref{main-algorithm} will be influenced by how many $2$-simplices are in $\dvr_{\bullet}^{1}(X)$. We make a conjecture as to the number of $2$-simplices required for point clouds that arise as samples from a manifold. 

    \begin{conjecture}
    Consider a finite sample of points $X$ from a $k$-dimensional manifold $M$ embedded in Euclidean space. Then $\dvr_{\infty}^{1}(X)$ consists of $O(kn)$ simplices. 
    \end{conjecture}

    The graphs in Figure \ref{no-2-simplices-graphs} seem to support this conjecture. 
\section{Higher Degree Equivalents of the Relative Nieghborhood Graph}

The relative Neighborhood Graph, first introduced in \cite{ToussaintRNG},  is a useful tool for finding the number of non-apparent homology classes in $\ph{1}(X)$ and is also useful for finding representatives of homology classes in $\ph{1}(X)$ \cite{koyama2024fastercomputationdegree1persistent}. In this section we present higher degree equivalents to $\rng(X)$. We will also present a more useful and easier to visualise ``clipped" version of these higher degree equivalents. 

\subsection{Relative Neighborhood Complexes}
    The Relative Neighborhood Graph of $X$ is the set of $1$-simplices with empty lune. We extend this definition in the most obvious fashion.

    \begin{definition}[Degree-$q$ Relative Neighborhood Complex]
        Consider a metric space $X$ and let $\vr_{\infty}^{q}(X)$ be the set of $q$-simplices which appear in the Vietoris-Rips filtration. Let $\rnc_{q}(X) \subset \vr_{\infty}^{q}(X)$ be the set of $q$-simplices $\sigma$ such that $\lune(\sigma) = \emptyset$. 
    \end{definition}

    \begin{remark}
        $\rnc_{1}(X)$ is the same as $\rng(X)$. 
    \end{remark}

    This Definition is certainly the most natural extention of $\rng(X)$, however unlike $\rng(X)$ isn't really insightful.
    In order to obtain more insightful extensions of $\rng(X)$ it is helpful to think of a modified version of $\rng(X)$. 

    \begin{definition}[Clipped Relative Neighborhood Graph]
        Consider a finite metric space $X$ and let $\mu_{1}$ be the acyclic partial matching defined on $\rvr_{\bullet}^{1}(X)$. Recall that $R_{1}(\mu_{1})$ is the set of critical $1$-simplices. Then we denote the following set of $1$-simplices 

        \begin{equation}
            \crng(X) := R_{1}(\mu_{1}) \cap \dvr_{\infty}^{1}(X)
        \end{equation}
as the ``clipped" relative neighborhood graph. 
    \end{definition}

We can easily extend the notion of clipped Relative Neigborhood Graphs to that of clipped relative neighborhood complexes. 

\begin{definition}[Clipped Relative Neighborhood Complexes]
    Consider a finite metric space $X$ and let $\mu_{q}$ be the acyclic partial matching defined on $\rvr_{\bullet}^{q}(X)$. Recall that $R_{q}(\mu)$ is the set of critical $q$-simplices. Then we denote the following set of $q$-simplices

    \begin{equation}
        \crnc_{q}(X) := R_{q}(\mu_{q})\cap \dvr_{\infty}^{q}(X)
    \end{equation}
\end{definition}

\begin{remark}
    $\crnc_{1}(X)$ is the same as $\crng(X)$
\end{remark}

In Figure \ref{dragon-point-cloud} we show two point clouds with $\crnc_{1}(X)$ and $\crnc_{2}(X)$ shown.

\begin{figure}[h!]
    \begin{subfigure}[t]{.5\textwidth}
    \centering
    \includegraphics[scale = 0.35]{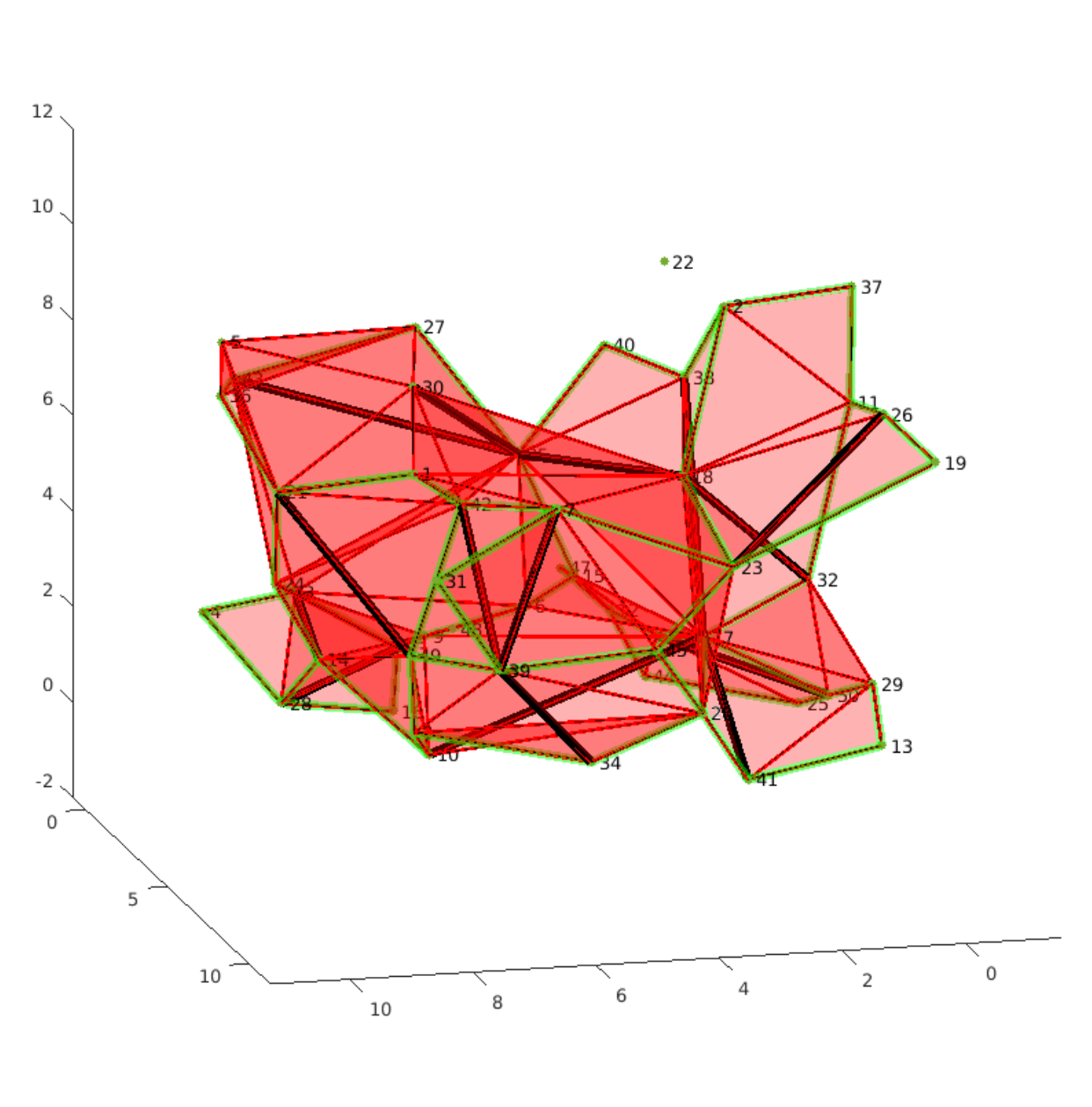}
    \subcaption{}
\end{subfigure}
\begin{subfigure}[t]{.5\textwidth}

    \centering
    \includegraphics[scale = 0.35]{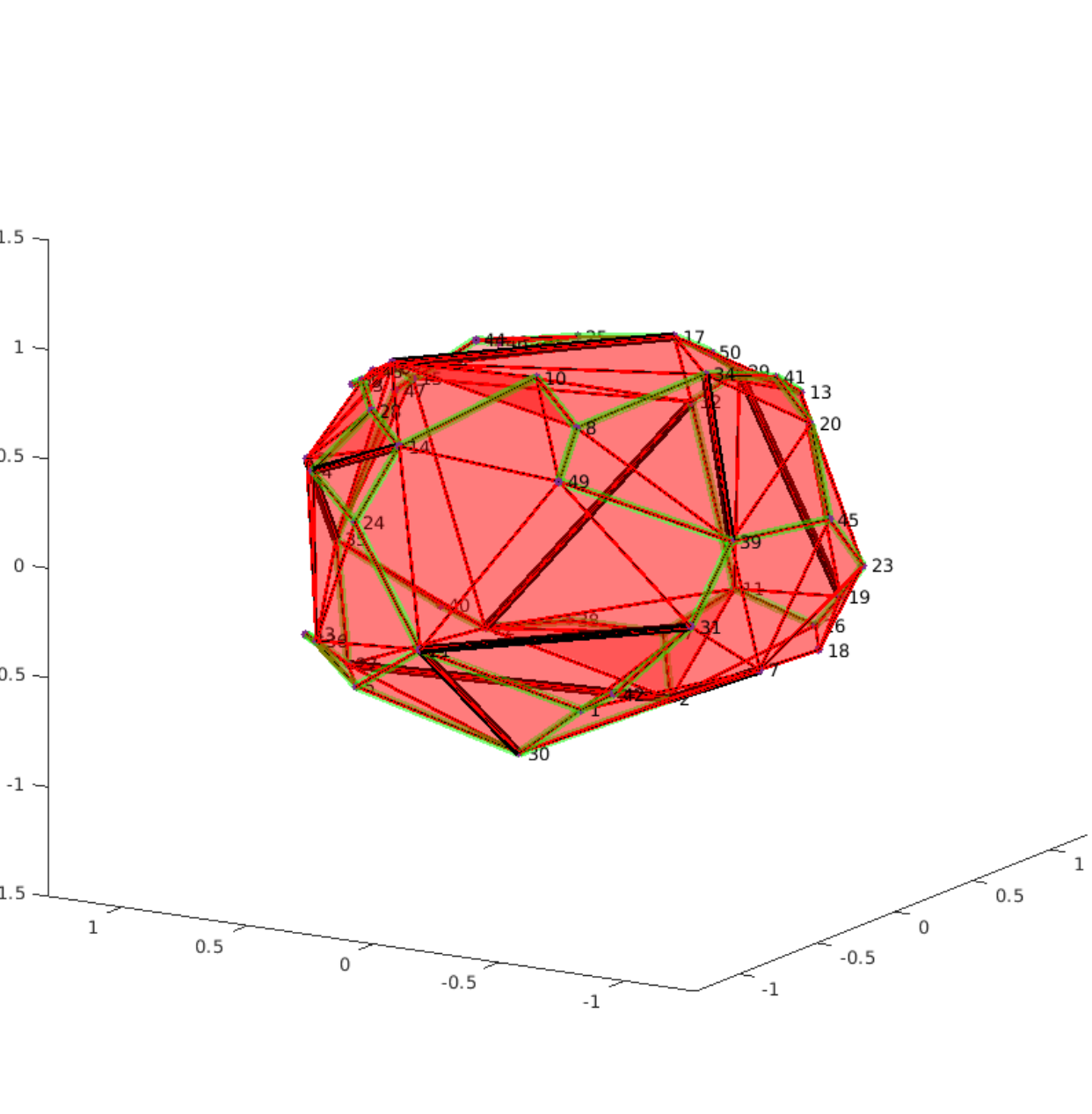}
 \subcaption{}
    \end{subfigure}
\caption{(a) $\crnc_{1}(X)$ and $\crnc_{2}(X)$ of 50 points uniformly distributed over a cube. $\crnc_{1}(X)$ is depicted in green and $\crnc_{2}(X)$ is depicted in red. (b) $\crnc_{1}(X)$ and $\crnc_{2}(X)$ of 50 points uniformly distributed over a cube. $\crnc_{1}(X)$ is depicted in green and $\crnc_{2}(X)$ is depicted in red.}
\label{dragon-point-cloud}
\end{figure}

The distilled Vietoris-Rips filtration uses considerably less 2-simplices. In the following, point clouds with $N \leq 700$ were generated using a uniform random sampling from a $10 \times  10 \times 10$ cube and then from a unit sphere. Figure \ref{no-2-simplices-graphs} shows that the number of $2$-simplices used to create $\dvr_{\infty}^{1}(X)$ varies linearly with the size of the point cloud over this range of $N$.  

\begin{figure}[h!]
    \begin{subfigure}[t]{.5\textwidth}
    \centering
    \includegraphics[scale = 0.5]{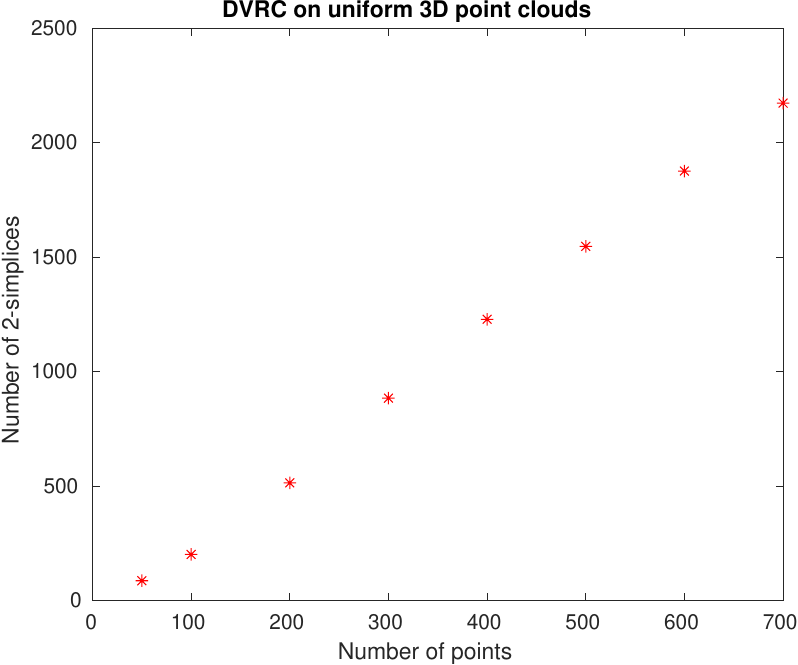}
    \subcaption{}
\end{subfigure}
\begin{subfigure}[t]{.5\textwidth}

    \centering
    \includegraphics[scale = 0.5]{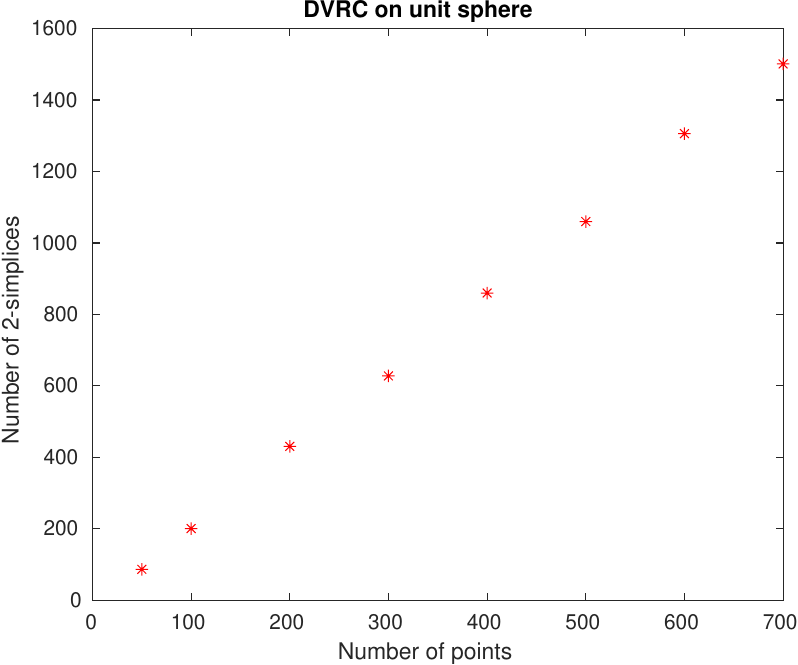}
 \subcaption{}
    \end{subfigure}
\caption{(a) Number of $2$-simplices used in $\dvr_{\infty}^{1}(X)$ (DVRC) on point clouds of size $50, 100, 200, 300, 400, 500, 600, 700$ generated from a uniform distribution in a $10\times 10\times 10$ cube. (b) Number of $2$-simplices used in $\dvr_{\infty}^{1}(X)$ (DVRC) on point clouds of size $50, 100, 200, 300, 400, 500, 600, 700$ generated on a unit sphere.}
\label{no-2-simplices-graphs}
\end{figure}

\newpage

\bibliography{references}
\bibliographystyle{ieeetr}

\end{document}